\documentclass[10pt]{article}
\usepackage{fullpage,lineno,graphicx,amsmath,amsthm,amssymb}

\usepackage{graphicx}
\usepackage{simplewick}
\usepackage{listings}
\usepackage{courier}
\newtheorem{theorem}{Theorem}[section]

\usepackage{caption}
\usepackage{subfig}
\numberwithin{figure}{section}
\usepackage{tcolorbox}
\usepackage{authblk}
\usepackage{hyperref}
\hypersetup{
    colorlinks=true,
    linkcolor=blue,
    filecolor=magenta,      
    urlcolor=cyan,
}
\numberwithin{equation}{section}

\author{Enzhi Li  \thanks{enzhililsu@gmail.com}}
\affil{Suning R \& D Center, Palo Alto, USA}

\title{Integral Equation Approach to Stationary Stochastic Counting Process with Independent Increments}

\date{\today}

\begin{document}

\maketitle

\noindent{\bf Abstract: }
Stationary stochastic processes with independent increments, of which the Poisson process is a prominent example, are widely used to describe real world events. With the basic assumption that a counting process is stationary and has independent increments, here I derive two integral equations to capture the time evolution of any such process. In order to solve these two integral equations explicitly, I need to introduce one more restriction condition. For sake of simplicity, I have imposed the Poisson condition on the two equations and successfully reproduced the renown Poisson results. The methods proposed here may also be applicable for investigating other stationary processes with independent increments. 
\vskip0.3cm


\section{Introduction}
Stationary stochastic counting process with independent increments is a widely used mathematical model that could capture some essence of the real-world happenings\cite{ross1996stochastic}. By definition, a counting process counts the number of events that occur during a time interval. For a counting process to be stationary, the probability of observing $n$ events during time interval $[t_1, t_2], t_2 \ge t_1 \ge 0$ depends only on the time difference $t_2 - t_1$, rather than on $t_1$ and $t_2$ separately, i.e., $p([t_1, t_2]_{n}) = f_{n}(t_2 - t_1)$. Here, we have introduced the notation $[t_1, t_2]_{n}$, which denotes the event that $n$ occurrences are observed during time interval $[t_1, t_2]$. A stationary process has independent increments if and only if the event $[t ,t + s]_{m}$ is independent of the event $[0, t]_{n}$, i.e., $p([0, t]_{n}, [t, t+s]_{m}) = p([0, t]_{n}) p([t, t+s]_{m}), m \ge 0, n \ge 0$. 

In this paper, I am going to derive two integral equations for any stationary stochastic counting process with independent increments. The first equation calculates the probability of observing $n$ events during time interval $[0, t]$, and the second equation calculates the probability density function of observing $n$ events during the time interval $[0, t]$. To solve these two equations, I need to introduce one more restriction condition. The introduction of this extra restriction specifies the type of counting process we are studying. For sake of simplicity, I will introduce restrictions that apply to Poisson process, and use the two integral equations derived here to reproduce the renown Poisson process\cite{ross1996stochastic}. The integral equations derived here are applicable to any stochastic counting process as long as it is stationary and has independent increments.

The organization of this paper is as follows. In section \ref{probability_function}, I will use the iterative principle to derive an integral equation for calculating the probability of observing $n$ events during the time interval $[0, t $. Solution of this integral equation requires the introduction of a restriction condition. I will introduce a restriction that is well motivated to derive the Poisson process. The equivalence of my restriction condition to the Poisson condition will be established in the same section. In section \ref{probability_density_function}, I will again use the iterative principle to derive an integral equation for the probability density function. Similarly, I need to introduce restriction condition to explicitly solve this integral equation. What distinguishes section \ref{probability_density_function} from section \ref{probability_function} is that the restriction condition in section \ref{probability_density_function} is more directly related with Poisson condition. Finally, I make a conclusion in section \ref{conclusion}. 

\section{An iterative method for deriving $p([0, t]_{n})$}
\label{probability_function}
In this section, I am going to derive an integral equation for $p([0, t]_{n})$, which is the probability of observing $n$ events during time interval $[0, t]$. Next, I will give an explicit solution to this equation by introducing one more assumption other than the assumption of stationarity and independent increments for the process. 

\subsection{Derivation of an integral equation for probability}
In order to establish an iterative integral equation, I will introduce a cutting procedure for the time interval $[0, t]$, which is
\begin{tcolorbox}
1. Select from the time interval $[0, t]$ a random number $\tau$ satisfying the condition that $0 \le \tau \le t$, and obtain two disjoin sub-time-intervals $[0, \tau]$ and $[\tau, t]$; \\
2. Demand that during time interval $[0, \tau]$, $k$ events occur, and during time interval $[\tau, t]$, $n - k$ events occur, where $0 \le k \le n$; \\
3. Let $\tau$ run the whole range $[0, t]$ and $k$ run the whole set $\{0, 1, 2, ..., n\}$. \\
\end{tcolorbox}

Every cutting involves three independent events, which are:
\begin{align}\nonumber
A &= [0, \tau]_{k}, \text{ which represents the occurrence of $k$ events during the time interval } [0, \tau];  \\\nonumber
B &=  [\tau, t]_{n - k}, \text{ which represents the occurrence of $n - k$ events during the time interval } [\tau, t]; \\\nonumber
C &=  \{\text{Selection of a random time $\tau$ from within the time interval } [0, t]\} \\\nonumber
\end{align}
Every cutting can be characterized by the time $\tau$ and the number $k$, and thus can be denoted as $U_{\tau}^{k} = [0, \tau]_{k} \cap [\tau, t]_{n - k}\cap \{\text{selection of time } \tau\}$. Due to the mutual independence of these three constituent sub-events, the probability of observing event $U_{\tau}^{k}$ is equal to 
\begin{eqnarray}
p(U_{\tau}^{k}) = p([0, \tau]_{k}) p([\tau, t]_{n - k}) p(\{\text{selection of time } \tau\})
\end{eqnarray}
The cutting events in the set $\{U_{\tau}^{k}\vert \tau \in [0, t], k = 0, 1, 2, ..., n\}$ are mutually exclusive due to the fact that the selection of time $\tau_1$ is exclusive to the selection of time $\tau_2$ as long as $\tau_1 \ne \tau_2$. Unfortunately, the selection of time $\tau$ occurs with probability 0, because a single point on the real axis has a zero measure. In order to avoid this, I need to first discretize the time interval $[0, t]$ into $N$ slices, each of which possesses a length $\Delta t = \frac{t}{N}$. Via this discretization procedure, we can put the cutting times $\tau$ on a lattice, and the selection of this discretized time $\tau$ gives us a finite probability. After discretization, the cutting procedure becomes 
\begin{tcolorbox}
1. Select at random from the set $\{1, 2, .., N-1\}$ a number $i$, and obtain two disjoin sub-time-intervals $[0, i\Delta t]$ and $[i\Delta t, t]$; \\
2. Demand that during time interval $[0, i\Delta t]$, $k$ events occur, and during time interval $[i\Delta t, t]$, $n - k$ events occur, where $0 \le k \le n$; \\
3. Let $i$ run over the whole set $\{1, 2, .., N-1\}$ and $k$ run over the whole set $\{0, 1, 2, ..., n\}$. \\
\end{tcolorbox}
Now every cutting event $U_{i}^{k} = [0, i\Delta t]_{k} \cap [i\Delta t, t]_{n - k}\cap \{\text{selection of index } i \}$ occurs with probability 
\begin{eqnarray}
p(U_{i}^{k}) &=& p([0, i\Delta t]_{k}) p([i\Delta t, t]_{n - k}) \frac{1}{N-1} \\\nonumber
&\approx& p([0, i\Delta t]_{k}) p([i\Delta t, t]_{n - k}) \frac{\Delta t}{t}
\end{eqnarray}
Since all the cutting events with distinct $i, k$ are mutually exclusive, the probability of observing $n$ events during time interval $[0, t]$ is the summation of all possible occurrence probabilities, that is, 
\begin{eqnarray}
p([0, t]_{n}) \approx \sum_{k = 0}^{n} \sum_{i = 1}^{N - 1} p([0, i\Delta t]_{k}) p([i\Delta t, t]_{n - k}) \frac{\Delta t}{t}
\end{eqnarray}
Setting $N \rightarrow \infty$, we have (remember that $f_{n}(t) = p([0, t]_{n})$)
\begin{eqnarray}
f_{n}(t) = \frac{1}{t}\sum_{k = 0}^{n}\int_{0}^{t} f_{k}(\tau) f_{n - k}(t - \tau) d\tau, n \ge 0
\label{probability_equation}
\end{eqnarray}
The above equation calculates the final probability by considering all possible values of $k$ and $\tau$, and can thus be considered to be a path integral over all past history. This method is analogous to the path integral formulation of quantum mechanics as proposed by Feynman\cite{feynman1948space, feynman2010quantum}. Because of this, I call this method the path integral approach to counting process. However, there are still two salient differences between the path integral formulated by Feynman and the path integral proposed here: 
\begin{itemize}
\item The path integral formulated by Feynman integrates over all possible paths, and each path is assigned a weight which is proportional to $e^{i\frac{S}{\hbar}}$, where $S = \int_{t_{i}}^{t_{f}} L dt$ is the Lagrangian action corresponding to that path. The path integral proposed here also integrates over all possible paths. However, an egalitarian principle is imposed and all the paths have equal weights.  
\item The path integral method as proposed here is applicable only to the stochastic counting processes that are stationary and possess independent increments. This is a strong restriction condition, and thus limit the applicability of this method. On the other hand, the path integral formulation as proposed by Feynman has a much wider application in the field of physics in that only fundamental principles of quantum mechanics are required for that method to be valid. I hope I can generalize the path integral approach proposed here in the foreseeable future. 
\end{itemize}
I have arrived at the above integral equation only from the assumption that the stochastic process is stationary and has independent increments. Explicit solution of this equation requires some extra assumptions. Next, I am going to introduce another assumption to give an explicit solution to the above equation. 

\subsection{A solution to the integral equation for probability}
Obviously, Equ. [\ref{probability_equation}] has a convolution form, and thus we can perform a Laplace transformation on this equation to solve it. Laplace transformation of this equation gives 
\begin{eqnarray}
-\frac{d}{dp}\hat{f}_{n}(p) = \sum_{k = 0}^{n} \hat{f}_{k}(p)\hat{f}_{n - k}(p), n \ge 0, 
\end{eqnarray}
where $\hat{f}_{n}(p) = \int_{0}^{\infty} f_{n}(t) e^{-pt}dt$ is the Laplace transformation of $f_{n}(t)$. 
Assume that $\int_{0}^{\infty} f_{n}(t) dt = \lambda^{-1}$ for $n \ge 0$. This assumption is equivalent to the condition that $\hat{f}_{n}(0) = \frac{1}{\lambda}, n \ge 0$. Later, I will show that this assumption is equivalent to the renown Poisson condition, which states that for a time interval that is short enough, the probability of observing one event is proportional to the length of the time interval, and no two events can occur simultaneously. 

I am going to use mathematical induction to prove the following theorem. 
\begin{theorem}
\begin{eqnarray}
\hat{f}_{n}(p) = \frac{\lambda^n}{(p + \lambda)^{n + 1}}, n \ge 0
\end{eqnarray}
\end{theorem}

\begin{proof}
When $n = 0$, we have 
\begin{eqnarray}
-\frac{d}{dp}\hat{f}_{0}(p) = \hat{f}^2_{0}(p)
\end{eqnarray}
From condition that $\hat{f}_{0}(0) = \lambda^{-1}$, we have 
\begin{eqnarray}
\hat{f}_{0}(p) = \frac{1}{p + \lambda}
\end{eqnarray}
For an arbitrary $n \ge 1$, assume that 
\begin{eqnarray}
\hat{f}_{k}(p) = \frac{\lambda^{k}}{(p + \lambda)^{k+1}}, 0 \le k \le n - 1
\end{eqnarray}
Plug this into Equ. [\ref{probability_equation}], we have (for $n$ that is even):
\begin{align}
 -\frac{d}{dp}\hat{f}_{n}(p) &= 2\sum_{k= 0}^{\frac{n}{2} - 1} \hat{f}_{k} (p)\hat{f}_{n - k}(p) + \hat{f}_{\frac{n}{2}}^2(p) \\\nonumber
  &= 2\hat{f}_{0}(p)\hat{f}_{n}(p) + 2\sum_{k= 1}^{\frac{n}{2} - 1} \hat{f}_{k} (p)\hat{f}_{n - k}(p) + \hat{f}_{\frac{n}{2}}^2(p) \\\nonumber
   &= 2\hat{f}_{0}(p)\hat{f}_{n}(p) + (n-1)\frac{\lambda^n}{(p + \lambda)^{n+2}}
\end{align}
and for $n$ that is odd, we have:
\begin{align} 
-\frac{d}{dp}\hat{f}_{n}(p) &= 2\sum_{k= 0}^{\frac{n - 1}{2}} \hat{f}_{k} (p)\hat{f}_{n - k}(p) \\\nonumber
 &= 2\hat{f}_{0}(p)\hat{f}_{n}(p) +2\sum_{k= 1}^{\frac{n - 1}{2}} \hat{f}_{k} (p)\hat{f}_{n - k}(p) \\\nonumber
  &= 2\hat{f}_{0}(p)\hat{f}_{n}(p) + (n-1)\frac{\lambda^n}{(p + \lambda)^{n+2}}
 \end{align}
Therefore, for any $n \ge 1$, we have 
\begin{eqnarray}
-\frac{d}{dp}\hat{f}_{n}(p) = 2\hat{f}_{0}(p)\hat{f}_{n}(p) + (n-1)\frac{\lambda^n}{(p + \lambda)^{n+2}}
\end{eqnarray}
The above equation can be equivalently transformed into the integrable form: 
\begin{eqnarray}
\frac{d}{dp}\Big( (p + \lambda)^2 \hat{f}_{n}(p)\Big) = -(n - 1)\frac{\lambda^n}{(p + \lambda)^{n}}
\end{eqnarray}
Integration of the above equation yields (remember that $\hat{f}_{n}(0) = \lambda^{-1}$)
\begin{eqnarray}
\hat{f}_{n}(p) = \frac{\lambda^{n}}{(p + \lambda)^{n+1}}
\end{eqnarray}
\end{proof}
Up to now, I have completed the proof of the above theorem. The probability function can be obtained by inverting the Laplace transform, and finally we have 
\begin{eqnarray}
f_{n}(t) &=& \frac{1}{2\pi i}\oint \frac{\lambda^{n}}{(p + \lambda)^{n+1}} e^{pt} dp \\\nonumber
&=& \frac{(\lambda t)^n}{n!} e^{-\lambda t}
\end{eqnarray}
Therefore, we have arrived at the Poisson process. 

\subsection{Equivalence of Poisson condition and the condition that $\hat{f}_{n}(0) = \lambda^{-1}, n \ge 0$}
For any stationary stochastic counting process with independent increments, $f_{n}(t)$ gives the probability of observing $n$ events during time interval $[0, t]$, with $n \ge 0$. This probability can also be interpreted in this way: what is the probability of waiting for a period of $t$ until all of the $n$ events have occurred? Once all the events have occurred, the process is considered to have terminated. The probability for the $n^{\text{th}}$ event to occur (or the process terminated) exactly at time $t$ is 0, since the set $\{ \text{time} = t\}$ has zero measure. However, the probability of observing the occurrence of the $n^{\text{th}}$ event in the short time interval $[t, t + \Delta t]$ is finite, and is equal to 
\begin{eqnarray}
p(t) = f_{n-1}(t)f_{1}(\Delta t)
\end{eqnarray}
Here, we have already implicitly invoked the assumption that the counting process is stationary and has independent increments, since for the process to terminate in the time interval $[t, t + \Delta t]$, we must have already observed the occurrence of $n - 1$ events before time $t$, which gives the factor $f_{n-1}(t)$, and we must also observe the occurrence of one more event during the time interval $[t, t + \Delta t]$, which gives the factor $f_{1}(\Delta t)$. Assume that the probability of observing one event during a short time period of length $\Delta t$ is equal to $\lambda \Delta t$(this is precisely the Poisson condition), we can convert the above probability into the form
\begin{eqnarray}
p(t) = f_{n - 1}(t) \lambda \Delta t
\end{eqnarray}
The termination of this counting process must occur at some time, and the termination of the process at time $t_1$ is exclusive to the termination of the process at another time $t_2$, we thus have the following integration: 
\begin{eqnarray}
\int_{0}^{\infty} f_{n - 1}(t) \lambda dt = 1, n \ge 1
\end{eqnarray}
When $n = 0$, $f_{n}(t)$ gives the probability of waiting for a time period of length $t$ until seeing an event occur. The average waiting time is given by $\int_{0}^{\infty} f_{0}(t) dt$. For a stationary stochastic process, the average waiting time between any two events should be a constant. Therefore, we have arrived at this condition: 
\begin{eqnarray}
\int_{0}^{\infty} f_{n}(t) dt = \lambda^{-1}, n \ge 0. 
\end{eqnarray}
This is just the condition we have employed to solve the integral equation [\ref{probability_equation}]. 

\section{An iterative method for deriving probability density function}
\label{probability_density_function}
In the previous section, I have derived and solved an integral equation for the probability function. In this section, I am going to derive and solve an integral equation for probability density function. I will still derive a general integral function that is valid for any stationary stochastic counting process with independent increments, and then specifies one more restriction condition to find an explicit solution to the general integral equation. 

\subsection{Derivation of a general integral equation for probability density function}
Let $N_t$ be a random variable that counts the number of events during the time interval $[0, t]$. First, we have the identity that 
\begin{align}
N_{t} = N_{t} + N_{t-\tau} - N_{t - \tau}, \tau \in [0, t]
\end{align}
Since we have assumed that the process is stationary, we thus have $N_{\tau} = N_{t} - N_{t - \tau}$. As a result, the above identity can be equivalently converted to 
\begin{eqnarray}
N_{t} = N_{t - \tau} + N_{\tau}
\end{eqnarray}
We have also assumed that the process has independent increments, which means $N_{\tau}$ and $N_{t - \tau}$ are independent. Denote the probability density function for the random variable $N_{t}$ as $h_{t}(x)$, we have a convolution equation: 
\begin{eqnarray}
h_{t}(x) = \int_{0}^{x} h_{\tau}(\xi) h_{t - \tau}(x - \xi) d\xi
\label{density_convolution}
\end{eqnarray}
Laplace transformation of the convolution yields 
\begin{eqnarray}
\hat{h}_{t}(p) = \hat{h}_{\tau}(p) \hat{h}_{t - \tau}(p)
\end{eqnarray}
Next, I consider the above equation as a functional equation for $g(t; p) = \hat{h}_{t}(p)$, with $t$ as the variable and $p$ as the parameter, and I have the following equation: 
\begin{eqnarray}
g(t; p) = g(\tau; p) g(t - \tau; p)
\label{functional}
\end{eqnarray}
Setting $\tau = 0$, we have 
\begin{eqnarray}
g(t; p) = g(0; p) g(t; p)
\end{eqnarray}
In order for Equ. [\ref{functional}] to have a nontrivial solution, we must have $g(0; p) = 1$. Taking the derivative of Equ. [\ref{functional}] with respect to $\tau$, and then setting $\tau = 0$, we have 
\begin{eqnarray}
0 = g^{\prime}(0; p) g(t; p) - g(0; p) g^{\prime}(t; p)
\end{eqnarray}
Or equivalently (we have used the fact that $g(0; p) = 1$), 
\begin{eqnarray}
\frac{dg(t; p)}{dt} = g^{\prime}(0; p) g(t; p)
\end{eqnarray}
Solution of the above equation yields 
\begin{eqnarray}
g(t; p) = c_{p} e^{g^{\prime}(0; p) t}
\end{eqnarray}
Again, due to the fact that  $g(0; p) = 1$, we have $c_{p} = 1$. Finally, I have obtained $g(t; p)$ as 
\begin{eqnarray}
g(t; p) = e^{g^{\prime}(0; p) t}
\end{eqnarray}
This equation relates the value of function at any time $t \ge 0$ with the derivative of the function at time $t = 0$ via an exponential map, and is thus analogous to the relation between Lie group and Lie algebra. In this sense, $g(t; p)$ is analogous to a single parameter Lie group, and $g^{\prime}(0; p)$ is analogous to its Lie algebra. We can thus call $g^{\prime}(0; p)$ the generator for the counting process. 
In order to find an explicit solution to Equ. [\ref{density_convolution}], I need to introduce one more restriction condition, which will be the topic of next subsection. 

\subsection{Solution of the general integral equation for one specific case}
Remember that 
\begin{eqnarray}
g(t; p) = \int_{0}^{\infty} h_{t}(x) e^{-px} dx
\end{eqnarray}
Thus, we have 
\begin{eqnarray}
g^{\prime}(t; p) = \int_{0}^{\infty} dx e^{-px} \frac{dh_{t}(x)}{dt}
\end{eqnarray}
Setting $t = 0$ gives 
\begin{eqnarray}
g^{\prime}(0; p) = \int_{0}^{\infty} dx e^{-px} \frac{dh_{t}(x)}{dt}\Bigg\vert_{t = 0}
\end{eqnarray}
If we make the assumption that for $t \rightarrow 0^{+}$, the probability follows this rule (this is the Poisson condition): 
\begin{eqnarray}
p(N_{t} = 0) &=& 1 - \lambda t, \\\nonumber
p(N_{t} = 1) &=& \lambda t, \\\nonumber
p(N_t \ge 2) &=& 0, 
\end{eqnarray}
then the probability density function for an infinitesimal $t$ is 
\begin{eqnarray}
h_{t}(x) = (1 - \lambda t)\delta(x) + \lambda t \delta(x - 1)
\end{eqnarray}
Thus, we have 
\begin{eqnarray}
g^{\prime}(0; p) &=& \int_{0}^{\infty} dx e^{-px} \frac{dh_{t}(x)}{dt}\Bigg\vert_{t = 0} \\\nonumber
&=& \int_{0}^{\infty} dx e^{-px} \Big( -\lambda \delta(x) + \lambda\delta(x - 1)\Big) \\\nonumber
&=& \lambda(-1 + e^{-p})
\end{eqnarray}
Finally, we have found the Laplace transformation of the probability density function at any time $t \ge 0$, which is 
\begin{eqnarray}
g(t; p) = e^{\lambda t(-1 + e^{-p})}
\end{eqnarray}
Comparison with the Laplace transform of the probability density function for Poisson distribution reveals that $h_{t}(x)$ is the probability density function for Poisson distribution with parameter $\lambda t$. 

\section{Conclusion}
\label{conclusion}
In this paper, I have derived two integral equations that are valid for any stationary stochastic counting process with independent increments. The first integral equation aims to calculate the probability of observing $n$ events during the time interval $[0, t]$, and the second integral equation aims to calculate the probability density function for a random variable $N_{t}$ which counts the number of events that occur within the time interval $[0, t]$. With the introduction of an extra condition, explicit solutions are found for each integral equation.  

\bibliography{ref}
\bibliographystyle{abbrv}

\end{document}